\documentclass[11pt, twoside, a4paper]{article}

\usepackage{amsmath}     				
\usepackage{amssymb}   					
\usepackage{amsthm}     				
\usepackage{graphicx}    				
\usepackage{hyperref}
\usepackage{textcomp}    				
\usepackage[T1]{fontenc} 				
\usepackage{marvosym}    				
\usepackage[sc]{mathpazo}  	 			
\usepackage[numbers,sort&compress]{natbib}

\usepackage{authblk} 					
\usepackage[usenames]{xcolor}  			
\usepackage{lastpage} 					

\usepackage{enumerate}	 				

\usepackage{hyperref} 					

\definecolor{ForestGreen}{rgb}{0.15,0.416,0.18}
\definecolor{EgyptBlue}{rgb}{0.063,0.2,0.65}
\hypersetup{
	colorlinks=true,
	linkcolor=EgyptBlue,
   citecolor=EgyptBlue,
   urlcolor=ForestGreen
}

\newtheorem{theorem}{Theorem}[section]

\newtheorem{lemma}[theorem]{Lemma}
\newtheorem{prop}[theorem]{Proposition}

\theoremstyle{definition}

\theoremstyle{definition}
\newtheorem{remark}[theorem]{Remark}
\theoremstyle{definition}

\numberwithin{equation}{section}
\numberwithin{table}{section}
\numberwithin{figure}{section}

\title{Title of the paper on differential equations}

\author[1]{\textbf{First Author}}
\author[1, 2]{\textbf{Second Author}\footnote{Corresponding author. Email: secondauthor@firstuniversity.com}}
\author[1]{\textbf{Third Author}}
\author[2]{\textbf{Fourth Author}}
\author[2, 3]{\textbf{Fifth Author}}

\affil[1]{First University, 1 University Street, City Name, H--1234, Country Name}
\affil[2]{Second University, 2 University Square, City Name H--9876, Country Name}
\affil[3]{Institute of Mathematics, Third University, 3 University Road, City Name, H--9888, Country Name}

\newcommand\shorttitle{The short title not exceeding 65 characters}

\newcommand\authorsshort{F. Author, S. Author, T. Author, F. Author and F. Author}

\newcommand{\ejqtdelogo}{tink.png}

\newcommand{\doi}[1]{\url{https://doi.org/#1}}

\makeatletter
\renewcommand{\maketitle}{\bgroup\setlength{\parindent}{0pt}

\begin{picture}(20,20)
     \put(-50,-5){\includegraphics[width=2.9truecm]{\ejqtdelogo}}
  \end{picture}

\vspace{1truecm}
\begin{center}{\vbox{\titlefont\@title}}\end{center}
\vspace{0.5truecm}
\begin{center}{\@author} \end{center}

\egroup
}

\renewcommand{\@fnsymbol}[1]{%
    \ifcase#1 \or {\,\Letter\!} \or\textasteriskcentered\or \textasteriskcentered\textasteriskcentered
    \else\@ctrerr\fi}
\makeatother

\newcommand*{\titlefont}{\fontsize{18}{21.6}\selectfont\textbf}

\makeatletter
\renewcommand\@author{\ifx\AB@affillist\AB@empty\AB@author\else
      \ifnum\value{affil}>\value{Maxaffil}\def\rlap##1{##1}%
    \AB@authlist\\[\affilsep]\vbox{\AB@affillist}
    \else  \AB@authors\fi\fi}
\makeatother

\makeatletter
\def\Year{20XX}
\def\IssueNumber{XX}
\def\ps@plain{
\def\@oddhead{\ifnum\thepage=1\hss\baselineskip8pt
\vtop to 0 pt{\vskip-0.4truecm\noindent\hbox{\hspace{1.3truecm}\Large
Electronic Journal of Qualitative Theory of Differential Equations\hss\linebreak}%
\vskip 0.1truecm
\noindent\hbox{\hspace{1.3truecm}\footnotesize \Year, No.\ {\bf \IssueNumber}, {1--\begin{NoHyper}\pageref{LastPage}\end{NoHyper}};\enspace \href{https://doi.org/10.14232/ejqtde.\Year.1.\IssueNumber}{https:/\!/doi.org/10.14232/ejqtde.\Year.1.\IssueNumber}} \hfill \hbox{\footnotesize \href{https://www.math.u-szeged.hu/ejqtde/}{www.math.u-szeged.hu/ejqtde/}}
\vss}
\else
\hss\textit{\shorttitle}\hss\hbox to 0pt{\hss\thepage}\fi}\def\@oddfoot{}
\def\@evenhead{\hbox to 0pt{\thepage\hss} \hss\textit{\authorsshort}\hss}
\def\@evenfoot{}
}
\makeatother

\begin{document}

\thispagestyle{empty}

\noindent{\large\bf Estimates of complex eigenvalues and an inverse spectral problem for the transmission eigenvalue problem}\footnote{This paper has been published in Electronic Journal of Qualitative Theory of Differential Equations, 2019, No. 38, 1-15; https://doi.org/10.14232/ejqtde.2019.1.38}
\\

\noindent {\bf  Xiao-Chuan Xu$^{1}$} {\bf, Chuan-Fu Yang$^2$}
{\bf, Sergey A. Buterin$^3$}
{\bf and Vjacheslav A. Yurko$^3$}

\noindent \emph{\small{$^1$School of Mathematics and Statistics, Nanjing University of Information Science and Technology, Nanjing, 210044, Jiangsu,
People's Republic of China} }

\noindent \emph{\small{$^2$Department of Applied
Mathematics, School of Science, Nanjing University of Science and Technology, Nanjing, 210094, Jiangsu,
People's Republic of China} }

\noindent \emph{\small{$^3$Department of Mathematics, Saratov state University, Astrakhanskaya 83, Saratov 410012, Russia} }
\\

\noindent\small{Email: xcxu@nuist.edu.cn, chuanfuyang@njust.edu.cn, buterinsa@info.sgu.ru, \\
 yurkova@info.sgu.ru}
\\

\noindent{\bf Abstract.}
{This work deals with the interior transmission eigenvalue problem: $y'' + {k^2}\eta \left( r \right)y = 0$ with boundary conditions ${y\left( 0 \right) = 0 = y'\left( 1 \right)\frac{{\sin k}}{k} - y\left( 1 \right)\cos k},$ where the function $\eta(r)$ is positive. We obtain the asymptotic distribution of non-real transmission eigenvalues under the suitable assumption for the square of the index of refraction $\eta(r)$. Moreover, we provide a uniqueness theorem for the case $\int_0^1\sqrt{\eta(r)}dr>1$, by using all transmission eigenvalues (including their multiplicities) along with a partial information of $\eta(r)$ on the subinterval. The relationship between the proportion of the needed transmission eigenvalues and the length of the subinterval on the given $\eta(r)$ is also obtained.



\section{Introduction and main results}

Consider the  interior transmission problem
\begin{eqnarray}\label{1}
{y'' + {k^2}\eta \left( r \right)y = 0,\quad 0 < r < 1},\quad {y\left( 0 \right) = 0 = y'\left( 1 \right)\frac{{\sin k}}{k} - y\left( 1 \right)\cos k},
\end{eqnarray}
where the square of the index of refraction $\eta(r)$ is a positive function in $ W_2^2[0,1]$ with the natural assumption $\eta(1)=1$ and $\eta'(1)=0$. The $k^2$-values for which the problem (\ref{1}) has a nontrivial solution $y\left(r\right)$ are called \emph{transmission eigenvalues}. The problem (\ref{1}) appears in the inverse scattering theory for a spherically stratified medium, which consists in determining the function $\eta(r)$ from transmission eigenvalues. To study the inverse spectral problem, one has to investigate the property of transmission eigenvalues, such as, the existence of real or non-real eigenvalues and their asymptotic distribution.

We introduce two key quantities. Denote
\begin{equation}\label{i1}
a := \int_0^1 {\sqrt {\eta \left( r  \right)} } dr,
\end{equation}
 which is explained physically as the time needed for the wave to travel from $r = 0$ to $r = 1$. Introduce the \emph{characteristic function}
\begin{equation}\label{c9}
d\left( k \right): = y'\left( 1,k \right)\frac{{\sin k}}{k} - y\left( 1,k \right)\cos k,
\end{equation}
where $y\left( {r,k} \right)$ is the solution of $y'' + {k^2}\eta \left( r \right)y = 0$ with the initial conditions $y\left( {0,k} \right) = 0$ and $y'\left( {0,k} \right) = 1$.
Obviously, the transmission eigenvalues coincide with the squares of zeros of $d\left( k \right)$.

For the asymptotic  behavior of the transmission eigenvalues, McLaughlin and Polyakov \cite{JR} first showed that if $a\neq1$ then there are infinitely many real eigenvalues $\{(k_n')^2\}_{n\ge n_0}$, which have the asymptotics
 \begin{equation}\label{zx}
   (k_n')^2=\frac{n^2\pi^2}{(a-1)^2}+\frac{1}{a-1}\int_0^aq(x)dx+\kappa_n,\quad \{\kappa_n\}\in l^2\quad n\to\infty,
 \end{equation}
 where $q(x)$ is defined in (\ref{2}). Some aspects of the asymptotics of large (real and non-real) transmission eigenvalues for the case $a=1$ were discussed in \cite{XC}.

 In 2015, Colton and co-authors \cite{DC2} studied the existence and distribution of the non-real transmission eigenvalues. They showed that if $a\neq 1$ and $\eta''(1)\ne0$ (this assumption can be weakened \cite{DC3}), then there exists infinitely many real and non-real transmission eigenvalues, moreover, the imaginary parts of the non-real eigenvalues go to infinity. In particular, they give an example to show the distribution of the transmission eigenvalues, which is
 \begin{equation*}
  \eta(r)=\frac{16}{(r+1)^2(r-3)^2}.
\end{equation*}
It is easy to calculate  $\eta(1)=1,\eta'(1)=0$ and $\eta''(1)=1\ne0$. For this $\eta(r)$, the distribution of the zeros of $d(k)$ in the right half plane is shown numerically in the Figure 1.1 (see \cite{DC2}).

\begin{center}
\begin{figure}[!hbt]
 \includegraphics[scale=0.45,angle=0]{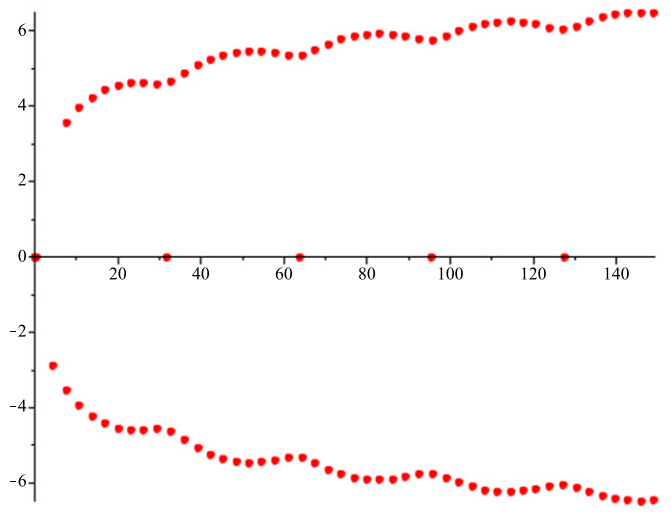}
  \caption{An example}
\end{figure}
\end{center}

From Figure 1, we see that the locations of the non-real zeros $\{x_n+iy_n\}$ of $d(k)$ in the right half-plane seem to satisfy asymptotically a logarithmic curve
$ y_n=\log (c x_n),$
where $c$ may be some complex number.
We will prove in theory that this is indeed true in the more general case (see Theorem \ref{t1}).

For the inverse spectral problem, many scholars contribute a lot of works (see \cite{TA2,TA1,BB,SA,SA1,CLH,DC1,JR,WS,WX,XC1,XC,XCa,CF1} and the references therein). Specifically, Aktosun and co-authors \cite{TA2,TA1} proved the uniqueness theorems and provided reconstruction algorithms for the cases $a<1$ and $a=1$. In the case $a=1$, to
determine the index of refraction uniquely, one has to know all the transmission eigenvalues  (including their
multiplicities) and
either a certain constant \cite{TA2,SA} or some knowledge of the $\eta(r)$ at $r=1$ \cite{WX,XC1}.
For the case $a>1$, however, there are only a few results. It is known \cite{DC1,JR} that the determination of $\eta(r)$ on $[0,1]$ with $\eta(1)=1$ and $\eta'(1)=0$ is equivalent to the determination of $q(x)$ on $[0,a]$ defined in (\ref{2}).
McLaughlin and Polyakov \cite{JR} first showed that if $a>1$ and $\eta(r)$ is known a priori on a subinterval $[\varepsilon_1,1]$ with $\varepsilon_1$ satisfying
\begin{equation}\label{xjaz}
\int_{\varepsilon_1}^1\sqrt{\eta(r)}dr=\frac{a+1}{2},
\end{equation}
then
$\eta(r)$ on $[0,\varepsilon_1]$  is uniquely determined by the transmission eigenvalues $\{(k_n')^2\}_{n\ge1}$ satisfying (\ref{zx}), where $\{(k_n')^2\}_{n=1}^{n_0-1}$ may be non-real.
In 2013, Wei and Xu \cite{WX} suggested to specify all transmission eigenvalues (including their multiplicities) and the norming constants, corresponding to the real eigenvalues, to obtain the unique determination of $\eta(r)$ on $[0,1]$.

In this paper, we will prove a new uniqueness theorem for the inverse spectral problem in the case $a>1$ (see Theorem \ref{t3}), by using the less known information on $\eta(r)$ and all eigenvalues (including real and non-real). Moreover, with the help of some ideas in \cite{RG,FB,MH,GR}, we give a relationship between the proportion of the needed eigenvalues and the length of the subinterval on the given $\eta(r)$ (see Theorem \ref{t4}).

The main results in this article are as follows.
\begin{theorem}\label{t1}
Assume that $\eta\in W_2^{m+3}[0,1]$ for some $m\in \mathbb{N}_0:=\{0\}\cup \mathbb{N}$. If $\eta(1)=1$, $\eta^{(u)}(1)=0$ for $u=\overline{1,m+1}$ and $\eta^{(m+2)}(1)\ne0$, then the characteristic function $d(k)$  has the non-real zeros $\{k_n^\pm\}$ satisfying the following asymptotic behavior, when $|n|\to\infty$, $n\in \mathbb{Z}$,

(i) $a\ne1$
\begin{equation*}
\begin{split}
   k_n^\pm=&\pm n\pi \pm \frac{i}{2}\log\left(\frac{4(2n\pi i)^{m+2}}{\eta^{(m+2)}(1)}\right)+\alpha_n^\pm,\;\;\alpha_n^\pm\in l^2\quad\text{for}\quad a>1,\\
     k_n^\pm=& \pm\frac{n\pi}{a} \pm \frac{i}{2a}\log\left(\frac{-4(2n\pi i)^{m+2}}{\eta^{(m+2)}(1)}\right)+\beta_n^\pm, \;\;\beta_n^\pm\in l^2\quad\text{for}\quad a<1.\\
\end{split}
\end{equation*}

(ii) $a=1$ and $\int_0^1q(x)dx\ne0$
\begin{equation*}
 k_n^\pm=\pm n\pi \pm\frac{i}{2}\log \left(\frac{-8(2n\pi i)^{m+1}\int_0^1q(s)ds}{\eta^{(m+2)}(1)}\right) +\gamma_n^\pm,\;\;\gamma_n^\pm\in l^2.
\end{equation*}
\end{theorem}

\begin{theorem}\label{t3}
Under the assumptions in Theorem \ref{t1}, if $a>1$ and $\eta(r)$ is known a priori on $[\varepsilon,1]$ with $\varepsilon$ satisfying
\begin{equation}\label{t31}
\int_\varepsilon^1\sqrt{\eta(r)}dr=\frac{a-1}{2},
\end{equation}
then $\eta(r)$ on $[0,1]$ is uniquely determined by all zeros of $d(k)$ (including multiplicity).
\end{theorem}

\begin{remark}
Eqs.(\ref{xjaz}) and (\ref{t31}) lead to $\int_{\varepsilon_1}^\varepsilon\sqrt{\eta(r)}dr=1$, which implies $\varepsilon>\varepsilon_1$.
\end{remark}

Let $N(r)$ be the number of non-real zeros $\{k_j\}_{j\ge1}$ of the function $d(k)$ in the disk $|k|\le r$, namely, $N(r):={\rm{\# }}\{j:|k_j|\le r\}$. From \cite{DC2,DC3} we see that if $a\ne1$ and $\eta(r)$ is non-constant near $r=1$ then the density of all zeros of $d(k)$  on the right half plane is ${(1+a)}/{\pi}$, and the density of the real zeros on the right half plane is ${|1-a|}/{\pi}$ if $a\ne1$. Note that $d(k)$ is an even function of $k$. It follows that if $a>1$ and  $\eta(r)$ is non-constant near $r=1$ then
\begin{equation}\label{xqa}
 N(r)=\frac{4r}{\pi}[1+o(1)],\quad r\to+\infty.
\end{equation}
Let $D$ be a subset of $\{k_j\}_{j\ge1}$, and denote  $N_D(r):={\rm{\# }}\{j:k_j\in D,|k_j|\le r\}$.

\begin{theorem}\label{t4}
Assume that $\eta\in C^2[0,1]$ with $\eta(1)=1$ and $\eta'(1)=0$, and $\eta(r)$ is non-constant near $r=1$. If $a>1$ and $\eta(r)$ is known a prior on $[\varepsilon_2,1]$ with $\varepsilon_2$
 satisfying
\begin{equation}\label{t40}
\int_{\varepsilon_2}^1\sqrt{\eta(r)}dr=b,\quad b>\frac{a-1}{2}
\end{equation}
then set $\{k_n'\}_{n\ge n_0}$ satisfying (\ref{zx}) and the subset $D$ satisfying $N_D(r)=\frac{2\alpha r}{\pi}[1+o(1)]$ as $r\to+\infty$ with $\alpha>a+1-2b$ uniquely determine $\eta(r)$ on $[0,1]$.
\end{theorem}
\begin{remark}
By virtue of (\ref{xqa}), we know that the value of $\alpha$ is at most $2$. Since $b>(a-1)/2$, we have $a+1-2b<2$. Thus the condition $\alpha>a+1-2b$ makes sense. Moreover, together with Theorems \ref{t3} and \ref{t4}, we see that if the known subinterval of $\eta(r)$ is a little bigger, then  infinitely many eigenvalues can be missing for the unique determination of $\eta(r)$.
\end{remark}
\section{Preliminaries}

In this section, we provide some known auxiliary results.

Using the Liouville transformation,
\begin{equation}\label{k0}
x {\rm{ = }}\int_0^r {\sqrt {\eta \left( \rho  \right)} d\rho },\quad  \varphi\left( x \right): = {\left( {\eta \left( r \right)} \right)^{\frac{1}{4}}}y\left( r \right),\quad r = r\left( x \right),
\end{equation}
we can write the equation $y'' + {k^2}\eta \left( r \right)y = 0$ with $y\left( {0,k} \right) = 0$ and $y'\left( {0,k} \right) = 1$ as
\begin{equation}\label{nj}
\varphi''(x) + \left( {{k^2} - q\left( x  \right)} \right)\varphi(x) = 0,\quad
\varphi\left( 0 \right) = 0,\quad \varphi'\left( 0 \right) = \eta {\left( 0 \right)^{ - \frac{1}{4}}},
\end{equation}
where
\begin{eqnarray}\label{2}
q\left( x\right) = \frac{\eta ''(r)}{4{{(\eta(r))}^2}} - \frac{5}{{16}}\frac{{{{\left(\eta'(r)\right)} }^2}}{{(\eta (r))}^3}.
\end{eqnarray}
Using the transformation operator theory (see, e.g. \cite{VM}), we have
\begin{equation}\label{gg}
 \eta(0)^{\frac{1}{4}} \varphi(x,k)=\frac{\sin (kx)}{k}+\int_0^xK(x,t)\frac{\sin(kt)}{k}dt,
\end{equation}
where $K(x,t)$ satisfies the following integral equation (see, e.g. \cite{SA})
\begin{equation}\label{k1}
\begin{split}
\! \!\!2K(x,t)=&\int_{\frac{x-t}{2}}^{\frac{x+t}{2}}q(\tau)d\tau+\int_{x-t}^xq(\tau)d\tau\int_{\tau+t-x}^\tau \!\!\!K(\tau,s)ds\\
 &+\int_{\frac{x-t}{2}}^{x-t}q(\tau)d\tau\int_{x-t-\tau}^\tau \!\!\!K(\tau,s)ds-\int_{\frac{x+t}{2}}^{x}q(\tau)d\tau\int_{x+t-\tau}^\tau\!\!\! K(\tau,s)ds,
\end{split}
\end{equation}
where $0\le t\le x\le a$. In particular, $2K(x,x)=\int_0^xq(s)ds$ and $K(x,0)=0$. On the other hand, from Eq.(1.2.9) in \cite{VM}, we know that
\begin{equation}\label{lkja}
  K(x,t)=K_0(x,t)-K_0(x,-t),
\end{equation}
where $K_0(x,t)$ with $0\le|t|\le x\le a$ satisfies that if $q\in C^m[0,a]$ then $K_0(x,\cdot)\in C^{m+1}[-x,x]$ for each fixed $x\in[0,a]$ (see Theorem 1.2.2 in \cite{VM}). It follows from (\ref{lkja}) that if $q \in C^m[0,a]$  then
\begin{equation}\label{kk}
\left. {\frac{ \partial^{2n} K(x,t)}{\partial t^{2n}}} \right|_{t=0}=0  \quad n=\overline{0,[(m+1)/2]},
\end{equation}
where $[(m+1)/2]$ denotes the entire part of $(m+1)/2$.

By virtue of (\ref{k0}) and $\eta(1)=1$ and $\eta'(1)=0$, we have $\varphi(a,k)=y(1,k)$ and $\varphi'(a,k)=y'(1,k)$. Thus,
\begin{eqnarray}\label{3}
y\left( 1,k \right) \!= \!\frac{1}{{{{{\eta \left( 0 \right)} }^{\frac{1}{4}}}}}\left[ {\frac{{\sin \left(ka\right)}}{k} \!-\!\frac{\cos(ka)}{2k^2}\int_0^a\!q(s)ds+ \int_0^{a} \!{K_t\left( {a,t} \right)\frac{{\cos \left( kt\right)}}{k^2}dt} } \right],
\end{eqnarray}
and
\begin{eqnarray}\label{4}
y'\left( 1,k \right) \!= \!{{\frac{1}{{\eta \left( 0 \right)^{\frac{1}{4}}}}} }\!\left[\! {\cos \left(ka\right) + \!\frac{{\sin \left(ka\right)}}{{2k}}\int_0^a \! {q\left( s \right)ds +\! \int_0^a \!\!{{K_x }\left( {a,t} \right)\frac{{\sin \left(kt\right)}}{k}dt} } } \!\right].
 \end{eqnarray}
Denote $K_1(t):=K_x(a,t)$ and $K_2(t):=K_t(a,t)$. Using Eq.(\ref{k1}), by tedious calculation, we have
\begin{equation}\label{k2}
  \begin{split}
 \!\!\!\!K_1(t)=\frac{1}{4}&\left[q\left(\frac{a+t}{2}\right)\!-\!q\left(\frac{a-t}{2}\right)\right]+\frac{1}{2}\int_{a-t}^aq(\tau)K(\tau,\tau+t-a)d\tau\\
 &-\frac{1}{2}\int_{\frac{a-t}{2}}^{a-t}\!q(\tau)K(\tau,a-\!t-\tau)d\tau\!+\!\frac{1}{2}\int_{\frac{a+t}{2}}^{a}\!q(\tau)K(\tau,a+t-\!\tau)d\tau,
  \end{split}
\end{equation}
and
\begin{equation}\label{k3}
  \begin{split}
 \!\!\!\! K_2(t)\!=\frac{1}{4}&\left[q\left(\frac{a+t}{2}\right)\!+\!q\left(\frac{a-t}{2}\right)\right]\!-\!\frac{1}{2}\int_{a-t}^aq(\tau)K(\tau,\tau+t-a)d\tau\\
 &+\frac{1}{2}\int_{\frac{a-t}{2}}^{a-t}\!q(\tau)K(\tau,a-t\!-\tau)d\tau+\frac{1}{2}\int_{\frac{a+t}{2}}^{a}\!q(\tau)K(\tau,a+t-\!\tau)d\tau.
  \end{split}
\end{equation}

To get Theorem \ref{t1}, we introduce the following transcendental equation
\begin{equation}\label{k5}
  z-\lambda\log z=w,
\end{equation}
where $\lambda$ is a constant in $\mathbb{C}$ and $\log z=\log|z|+i\arg z$ with $-\pi<\arg z\le \pi$.
\begin{prop}\label{p2.1}
The transcendental equation (\ref{k5}) has a unique solution
\begin{equation}\label{k6}
  z(w)=w+\lambda\log w+O\left(\frac{\log |w|}{|w|}\right)
\end{equation}
for any sufficiently large $|w|$.
\end{prop}
Using a similar discussion in \cite[p.50]{MF} or \cite{ST}, one can prove  Proposition \ref{p2.1}. For convenience of reader, we give the proof in the Appendix. We will transform the equation $d(k)=0$ to the equation with the form of (\ref{k5}), and then use (\ref{k6}) to obtain the asymptotics of non-real transmission eigenvalues.

 For the inverse spectral problem, we shall use the following three lemmas.
 \begin{lemma}\emph{(See \cite[p.28]{PK})}\label{l3}
Let $G(k)$ be analytic in $\mathbb{C}_+$ and continuous in $\overline{\mathbb{C}}_+:=\mathbb{C}_+\cup\mathbb{R}$. Suppose that

(i) $\log|G(k)|=O (k)$ for $|k|\to\infty$ in $\mathbb{C}_+:=\{k\in\mathbb{C}:{\rm Im}k>0\}$,

(ii) $|G(x)|\le C$ for some constant $C>0$, $x\in \mathbb{R}$,

(iii)$\mathop {\varlimsup}\limits_{\tau\to+\infty}{\log|G({\rm i}\tau)|}/{\tau}=A$.\\
Then, for $k\in\overline{\mathbb{C}}_+$, there holds
\begin{equation*}
  |G(k)|\le Ce^{A{\rm Im}k}.
\end{equation*}
\end{lemma}

\begin{lemma}\emph{(See \cite{GR})}\label{l0}
For an arbitrary $0<b<\infty$ and $ p(\cdot)\in L^2[0,b]$, if $$\int_0^bp(x)\varphi(x,k)\tilde{\varphi}(x,k)dx=0$$ for all $k>0$, then $p(x)=0$ on the interval $[0,b]$, where $\varphi(x,k) $ and $\tilde{\varphi}(x,k) $ are defined by (\ref{nj}) corresponding to $q$ and $\tilde{q}$, respectively.
\end{lemma}

\begin{lemma}[See Chapter \uppercase\expandafter{\romannumeral4} of \cite{BL}]\label{3.1}
For any entire function $g(k)\not\equiv0$ of exponential type, the following inequality holds,
 \begin{equation*}
\mathop {\varliminf }\limits_{r \to \infty }  \frac{N_g(r)}{r}\leq\frac{1}{2\pi}\int_0^{2\pi}h_g(\theta)d\theta,
 \end{equation*}
where $N_g(r)$ is the number of zeros of $g(k)$ in the disk $|k|\leq r\;(r>0)$ and $h_g(\theta):=\mathop {\varlimsup }\limits_{r \to \infty }\frac{\log |g(re^{i\theta})|}{r}$ with $k=re^{i\theta}$.
\end{lemma}

\section{Proofs}
\begin{proof}[Proof of Theorem \ref{t1}]
Rewrite Eqs.(\ref{3}) and (\ref{4}) as
\begin{equation}\label{k7}
 y(1,k)=\frac{\sin(ka)}{\eta(0)^{\frac{1}{4}}k}\left[1+P_1(k)\right],\quad  y'(1,k)=\frac{\cos(ka)}{\eta(0)^{\frac{1}{4}}}\left[1+P_2(k)\right],
\end{equation}
where
\begin{equation}\label{k8}
 P_1(k)=-\frac{\cot(ka)}{2k}\int_0^aq(s)ds+\frac{1}{k\sin(ka)}\int_0^aK_1(t)\cos(kt)dt,
\end{equation}
and
\begin{equation}\label{k9}
 P_2(k)=\frac{\tan(ka)}{2k}\int_0^aq(s)ds+\frac{1}{k\cos(ka)}\int_0^aK_2(t)\sin(kt)dt.
\end{equation}
By (\ref{c9}), we have
\begin{equation}\label{k11}
\begin{split}
\eta(0)^\frac{1}{4}d(k)&=\frac{\sin k}{k}\cos(ka)[1+P_2(k)]-\cos k\frac{\sin(ka)}{k}[1+P_1(k)]\\
&=\frac{\sin (k(1-a))}{2k}[2+P_2(k)+P_1(k)]+\frac{\sin(k(1+a))}{2k}[P_2(k)-P_1(k)].
\end{split}
\end{equation}
Now we shall estimate $P_2(k)-P_1(k)$ when $|k|\to\infty$ in $\mathbb{C}$.
Since $\eta\in W_2^{m+3}[0,1]$ with $\eta^{(u)}(1)=0$ for $u=\overline{1,m+1}$ and $\eta^{(m+2)}(1)\ne0$, it follows from (\ref{2}) that $q\in W_2^{m+1}[0,a]$ with $q^{(u)}(a)=0$ for $u=\overline{0,m-1}$ and $q^{(m)}(a)=\frac{\eta^{(m+2)}(1)}{4}\ne0$. Integrating by parts in (\ref{k8}) and (\ref{k9}) for $m+1$ times, and using (\ref{kk}), we have
\begin{subequations}\label{k12}
\begin{equation}
\begin{split}
\int_0^aK_1(t)\cos(kt)dt=&\sin(ka)\sum_{u=0}^s\frac{K_1^{(2u)}(a)}{(-1)^uk^{2u+1}}+\cos(ka)\sum_{v=0}^{s-1}\frac{K_1^{(2v+1)}(a)}{(-1)^vk^{2v+2}}\\
&+\frac{\varepsilon_1(k)}{k^{2s+1}},\quad \text{if\quad $m=2s$},\quad s\in\mathbb{N}_0,
\end{split}
\end{equation}
or
\begin{equation}
\begin{split}
\int_0^aK_1(t)\cos(kt)dt=&\sin(ka)\sum_{u=0}^s\frac{K_1^{(2u)}(a)}{(-1)^uk^{2u+1}}+\cos(ka)\sum_{v=0}^{s}\frac{K_1^{(2v+1)}(a)}{(-1)^vk^{2v+2}}\\
&+\frac{\varepsilon_2(k)}{k^{2s+2}},\quad \text{if\quad $m=2s+1$},\quad s\in\mathbb{N}_0,
\end{split}
\end{equation}
\end{subequations}
and
\begin{subequations}\label{k13}
\begin{equation}
\begin{split}
\!\!\!\int_0^aK_2(t)\sin(kt)dt\!=&\cos(ka)\sum_{u=0}^s\frac{K_2^{(2u)}(a)}{(-1)^{u+1}k^{2u+1}}+\sin(ka)\sum_{v=0}^{s-1}\frac{K_2^{(2v+1)}(a)}{(-1)^vk^{2v+2}}\\
&+\frac{\varepsilon_3(k)}{k^{2s+1}},\quad \text{if\quad $m=2s$},\quad s\in\mathbb{N}_0,
\end{split}
\end{equation}
or
\begin{equation}
\begin{split}
\!\!\!\int_0^aK_2(t)\sin(kt)dt\!=&\cos(ka)\sum_{u=0}^s\frac{K_2^{(2u)}(a)}{(-1)^{u+1}k^{2u+1}}\!+\!\sin(ka)\sum_{v=0}^{s}\frac{K_2^{(2v+1)}(a)}{(-1)^vk^{2v+2}}\\
&+\frac{\varepsilon_4(k)}{k^{2s+2}},\quad \text{if\quad $m=2s+1$},\quad s\in\mathbb{N}_0,
\end{split}
\end{equation}
\end{subequations}
where $\varepsilon_j(k)\;(j=\overline{1,4}) $ have the form of $\int_0^aK_0(t)\sin (kt)dt$ or $\int_0^aK_0(t)\cos (kt)dt$ with some $K_0(\cdot)\in L^2(0,a)$. We only discuss the case $m=2s$, and the case $m=2s+1$ is similar. Note that $\varepsilon_j(k)=o(e^{|{\rm Im}k|a}) $ as $|k| \to\infty $ in $\mathbb{C}$ (see \cite[p.15]{PT}). Substituting (\ref{k12}) and (\ref{k13}) into (\ref{k8}) and (\ref{k9}), respectively, and subtracting, we obtain
\begin{equation}\label{k14}
 \begin{split}
 P_2(k)-P_1(k)=&\frac{\int_0^aq(s)ds}{2k}[\tan(ka)+\cot(ka)]+\sum_{u=0}^s\frac{K_2^{2u}(a)+K_1^{2u}(a)}{(-1)^{u+1}k^{2u+2}}\\
 &+\tan(ka)\sum_{v=0}^{s-1}\frac{K_2^{(2v+1)}(a)}{(-1)^{v}k^{2v+3}}-\cot(ka)\sum_{v=0}^{s-1}\frac{K_1^{(2v+1)}(a)}{(-1)^{v}k^{2v+3}}\\
 &+\frac{\varepsilon_5(k)}{k^{2s+2}},\;\;\varepsilon_5(k)=o(1),\quad |k|\to\infty,\quad k\in \mathbb{C}_\pm,
 \end{split}
\end{equation}
where $\mathbb{C}_\pm:=\{k\in \mathbb{C}:\pm{\rm Im}k>0\}$. Note that for $|k|\to\infty$ in $\mathbb{C}_\pm$,
\begin{equation}\label{k15}
 \tan(ka)=\pm i+O(e^{-2a|{\rm Im}k|}), \quad \cot(ka)=\mp i+O(e^{-2a|{\rm Im}k|}).
\end{equation}
Substituting (\ref{k15}) into (\ref{k14}), and observing that $\tan(ka)+\cot(ka)=2/\sin(2ka)$, we get
\begin{equation}\label{k16}
 \begin{split}
 P_2(k)-P_1(k)=&\frac{\int_0^aq(s)ds}{k\sin(2ak)}+\sum_{u=0}^s\frac{K_2^{2u}(a)+K_1^{2u}(a)}{(-1)^{u+1}k^{2u+2}}\\
 &\pm i\sum_{v=0}^{s-1}\frac{K_2^{(2v+1)}(a)+K_1^{(2v+1)}(a)}{(-1)^{v}k^{2v+3}}+O\left(\frac{e^{-2a|{\rm Im}k|}}{k^3}\right)\\
 &+\frac{\varepsilon_5(k)}{k^{2s+2}},\quad |k|\to\infty,\quad k\in \mathbb{C}_\pm.
 \end{split}
\end{equation}
Now we shall calculate $K_1^{(u)}(a)+K_2^{(u)}(a)$ for $u=\overline{0,m}$.
Using (\ref{k2}) and (\ref{k3}), we have
\begin{equation*}
K(t):=K_1(t)+K_2(t)=\frac{1}{2}q\left(\frac{a+t}{2}\right)+\int_{\frac{a+t}{2}}^aq(\tau)K(\tau,a+t-\tau)d\tau.
\end{equation*}
Since $q^{(u)}(a)=0$ for $u=\overline{0,m-1}$ and $q^{(m)}(a)=\frac{\eta^{(m+2)}(1)}{4}\ne0$, we obtain
\begin{equation}\label{k17}
K^{(u)}(a)=0,\; u=\overline{0,m-1},\; K^{(m)}(a)=\frac{q^{(m)}(a)}{2^{m+1}}=\frac{\eta^{(m+2)}(1)}{2^{m+3}}.
\end{equation}
Substituting (\ref{k17}) into (\ref{k16}), we get, for the case $m=2s$,
\begin{subequations}\label{k18}
\begin{equation}
 \begin{split}
 P_2(k)-P_1(k)=&\frac{\int_0^aq(s)ds}{k\sin(2ak)}+\frac{(-1)^{\frac{m}{2}+1}\eta^{(m+2)}(1)}{2^{m+3}k^{m+2}}\\
 &+O\left(\frac{e^{-2a|{\rm Im}k|}}{k^3}\right)+\frac{\varepsilon_5(k)}{k^{m+2}},\quad |k|\to\infty,\quad k\in \mathbb{C}_\pm.
 \end{split}
\end{equation}
Similarly, one can get that for the case $m=2s+1$,
\begin{equation}
 \begin{split}
 P_2(k)-P_1(k)=&\frac{\int_0^aq(s)ds}{k\sin(2ak)}\pm i\frac{(-1)^{\frac{m-1}{2}}\eta^{(m+2)}(1)}{2^{m+3}k^{m+2}}\\
 &+O\left(\frac{e^{-2a|{\rm Im}k|}}{k^3}\right)+\frac{\varepsilon_5(k)}{k^{m+2}},\quad |k|\to\infty,\quad k\in \mathbb{C}_\pm.
 \end{split}
\end{equation}
\end{subequations}

Let $k:=\sigma+i\tau$, and consider the domain $$\mathbb{C}_\pm^\epsilon:=\left\{k\in \mathbb{C}_\pm:|\tau|\ge\frac{m+2-\epsilon}{2a}\log |\sigma|,0<\epsilon<1\right\}\quad \text{if}\quad a\ne1.$$
Substituting (\ref{k18}) into (\ref{k11}),
we have that if $a\ne1$ and $|k|\to\infty$ in $\mathbb{C}_\pm^\epsilon$, then, for the case $m=2s$,
\begin{subequations}\label{k19}
\begin{equation}\label{25a}
\eta(0)^\frac{1}{4}d(k)=\frac{\sin (k(1-a))}{k}\left[1+O\left(\frac{1}{k}\right)\right]+\frac{\eta^{(m+2)}(1)\sin(k(1+a))}{(-1)^{\frac{m}{2}+1}2(2k)^{m+3}}[1+\varepsilon_6(k)],
\end{equation}
and
for the case $m=2s+1$,
\begin{equation}
\eta(0)^\frac{1}{4}d(k)=\frac{\sin (k(1-a))}{k}\left[1+O\left(\frac{1}{k}\right)\right]\pm i\frac{\eta^{(m+2)}(1)\sin(k(1+a))}{(-1)^{\frac{m-1}{2}}2(2k)^{m+3}}[1+\varepsilon_6(k)],
\end{equation}
\end{subequations}
if $a=1$, $\int_0^1q(s)ds\ne0$ and $|k|\to\infty$ in $\mathbb{C}_\pm$, then for the case $m=2s$,
\begin{subequations}
\begin{equation}\label{36a}
\eta(0)^\frac{1}{4}d(k)=\frac{\int_0^1q(s)ds}{2k^2}\left[1+O\left(\frac{1}{k^2}\right)\right]+\frac{\eta^{(m+2)}(1)\sin(2k)}{(-1)^{\frac{m}{2}+1}2(2k)^{m+3}}[1+\varepsilon_7(k)],
\end{equation}
and
for the case $m=2s+1$,
\begin{equation}
\eta(0)^\frac{1}{4}d(k)=\frac{\int_0^1q(s)ds}{2k^2}\left[1+O\left(\frac{1}{k^2}\right)\right]\pm i\frac{\eta^{(m+2)}(1)\sin(2k)}{(-1)^{\frac{m-1}{2}}2(2k)^{m+3}}[1+\varepsilon_7(k)],
\end{equation}
\end{subequations}
where
\begin{equation}\label{sbw1}
 \varepsilon_6(k)=c|k|^{m+1}e^{-2a|{\rm Im}k|}+\varepsilon_5(k)=o(1),\quad |k|\to\infty,\;\;k\in \mathbb{C}_\pm^\epsilon,
\end{equation}
and
\begin{equation}\label{sbw2}
\varepsilon_7(k)=c |k|^{m-1}e^{-2|{\rm Im}k|}+\varepsilon_5(k)=o(1),\quad |k|\to\infty,\;\;k\in \mathbb{C}_\pm.
\end{equation}

The remaining proof should be divided into six subcases: (i) $a>1$ and $m=2s$; (ii) $a>1$ and $m=2s+1$; (iii) $a<1$ and $m=2s$; (iv) $a<1$ and $m=2s+1$; (v) $a=1$ and $m=2s$; (vi) $a=1$ and $m=2s+1$. We only discuss the subcases (i) and (v) in details, and the other cases are similar and omitted.

 Case (i): by virtue of (\ref{25a}), we know that $d(k)=0$ for $|k|\to\infty$ in $\mathbb{C}_\pm^\epsilon$ is equivalent to that
\begin{equation*}
2^{m+4}k^{m+2}\sin(k(1-a))\left[1+O\left(\frac{1}{k}\right)\right]=(-1)^{\frac{m}{2}}\eta^{(m+2)}(1)\sin(k(1+a))[1+\varepsilon_6(k)].
\end{equation*}
Setting $k=\frac{z}{i}$, we have $(-1)^{\frac{m}{2}}(\frac{1}{i})^{m+2}=(-1)^{\frac{m}{2}}(-1)^{\frac{m}{2}+1}=-1$, and furthermore,
\begin{equation*}
  \frac{2^{m+4}z^{m+2}}{\eta^{(m+2)}(1)}[e^{z(a-1)}-e^{z(1-a)}]=[e^{z(1+a)}-e^{-z(1+a)}][1+\varepsilon_6(k)],\;|z|\to\infty,\;k\in \mathbb{C}_\pm^\epsilon,
\end{equation*}
Taking logarithm on both sides of the above equation, we get that for sufficiently large $n\in \mathbb{Z}$,
\begin{equation*}
\left\{\begin{split}
z-\frac{m+2}{2}\log z=w_n,\quad w_n:=-n\pi i+\frac{1}{2}\log\left(\frac{2^{m+4}}{\eta^{(m+2)}(1)}\right)+\varepsilon_8(k),\;{\rm Re}z>0,\\
z+\frac{m+2}{2}\log z=w_n,\quad w_n:=n\pi i-\frac{1}{2}\log\left(\frac{2^{m+4}}{\eta^{(m+2)}(1)}\right)+\varepsilon_8(k),\;{\rm Re}z<0,
\end{split}\right.
\end{equation*}
where
\begin{equation}\label{xunm1}
\begin{split}
\varepsilon_8(k)&=\pm\log(1+\varepsilon_6(k))\pm\log(1+e^{2|{\rm Re}z|(1-a)})\pm\log(1+e^{-2|{\rm Re}z|(1+a)})\\
&=o(1),\quad|k|\to\infty,\quad k\in \mathbb{C}_\pm^\epsilon.
\end{split}
\end{equation}
It follows from (\ref{k5}) and (\ref{k6}) and $z=ik$ that
\begin{equation}\label{xul1}
   k_n^\pm=\pm n\pi \pm \frac{i}{2}\log\left(\frac{4(2n\pi i)^{m+2}}{\eta^{(m+2)}(1)}\right)+\alpha_n^\pm,\quad \alpha_n^\pm=o(1),\quad n\to\infty.
\end{equation}
Clearly, the above sequences belong to the domain $\mathbb{C}_\pm^\epsilon$ for all large $|n|$.

Substituting (\ref{xul1}) into (\ref{k12}) and (\ref{k13}), we get that $\varepsilon_j(k_n^\pm)e^{-a|{\rm Im }k_n^\pm|}\in l^2$ for $j=\overline{1,4}$, which implies $\varepsilon_5(k_n^\pm)\in l^2$. It follows from (\ref{sbw1}) and (\ref{sbw2}) that $\varepsilon_8(k_n^\pm)\in l^2$. Taking (\ref{k5}) and (\ref{k6}) into account, we can obtain $\alpha_n^\pm\in l^2$.

 Case (v): by virtue of (\ref{36a}),
we know that $d(k)=0$ for $|k|\to\infty$ in $\mathbb{C}_\pm$ is equivalent to that
\begin{equation*}
  \frac{\int_0^1q(s)ds}{\eta^{(m+2)}(1)}2^{m+3}k^{m+1}(-1)^\frac{m}{2}=\sin(2k)[1+\varepsilon_7'(k)],\quad |k|\to\infty,\quad k\in\mathbb{C}_\pm,
\end{equation*}
where $\varepsilon_7'(k)=\varepsilon_7(k)+O(k^{-2})$.
Setting $k=\frac{z}{i}$, we have $(-1)^{\frac{m}{2}}(\frac{1}{i})^{m+1}=(-1)^{\frac{m}{2}}(-1)^{\frac{m}{2}}\frac{1}{i}=\frac{1}{i}$, and
\begin{equation*}
  \frac{\int_0^1q(s)ds}{\eta^{(m+2)}(1)}2^{m+4}z^{m+1}=[e^{2z}-e^{-2z}][1+\varepsilon_7'(k)],\quad |z|\to\infty,\quad k\in\mathbb{C}_\pm,
\end{equation*}
which implies that for sufficiently large $n\in \mathbb{Z}$
\begin{equation*}
  \left\{ \begin{split}
 &z-\frac{m+1}{2}\log z=-n\pi i+\frac{1}{2}\log \frac{\int_0^1q(s)ds}{\eta^{(m+2)}(1)}2^{m+4} +\varepsilon_8'(k),\quad {\rm Re }z>0,\\
   & z+\frac{m+1}{2}\log z=n\pi i-\frac{1}{2}\log \frac{-\int_0^1q(s)ds}{\eta^{(m+2)}(1)}2^{m+4} +\varepsilon_8'(k),\quad {\rm Re }z<0.
  \end{split}\right.
\end{equation*}
It follows from (\ref{k5}) and (\ref{k6}) and $z=ik$ that for $n\in \mathbb{Z}$ and $|n|\to\infty$
\begin{equation*}
  \left\{ \begin{split}
 &k_n^-=-n\pi -\frac{i}{2}\log \left(\frac{\int_0^1q(s)ds}{\eta^{(m+2)}(1)}2^{m+4}(-n\pi i)^{m+1}\right) +\gamma_n^-,\\
&k_n^+=n\pi +\frac{i}{2}\log \left(\frac{-\int_0^1q(s)ds}{\eta^{(m+2)}(1)}2^{m+4}(n\pi i)^{m+1}\right) +\gamma_n^+.
  \end{split}\right.
\end{equation*}
Using a similar argument, one gets $\gamma_n^\pm\in l^2$.

Through similar arguments, one obtains asymptotics of other cases. The proof is finished.
\end{proof}

\begin{proof}[Proof of Theorem \ref{t3}]
Since the function $d(k)$ is an entire function of $k$ of order $1$ and even with respect to $k$, by Hadamard's  factorization theorem,
\begin{equation}\label{uy}
  d(k)=\gamma E(k),\quad E(k):=k^{2s}\prod_{k_n\ne0}\left(1-\frac{k^2}{k_n^2}\right),
\end{equation}
where $s$ is the multiplicity of the zero eigenvalue.

Using (\ref{i1}), (\ref{k0}) and (\ref{2}),  one can verify that specification of $\eta(r)$ on
$[\varepsilon,1]$ with $\varepsilon$ satisfying (\ref{t31}) is equivalent to
specification of $q(x)$ for $x\in[\frac{a+1}{2},a]$. Let us prove that $q(x)$ on $[0,a]$ is uniquely determined by $E(k)$ and the known $q(x)$ on $[\frac{a+1}{2},a]$. If it is true, then $\eta(r)$ on $[0,1]$ with $\eta(1)=1$ and $\eta'(1)=0$ is uniquely determined by $E(k)$ and the known $\eta(r)$ on $[\varepsilon,1]$. (See \cite{JR}).

 Suppose that there are two functions $q$ and $\tilde{q}$ corresponding to the same $E(k)$ defined by (\ref{uy}). Let ($a,\varphi$) and ($\tilde{a},\tilde{\varphi}$) be their corresponding quantities in (\ref{i1}) and (\ref{nj}). By virtue of (\ref{zx}) and $a>1$, we obtain
 \begin{equation*}
   a=\tilde{a}.
 \end{equation*}

 Denote
 \begin{equation}\label{msz}
   g(k):=\int_0^{\frac{a+1}{2}}[\tilde{q}(x)-q(x)]\varphi(x,k)\tilde{\varphi}(x,k)dx.
 \end{equation}
 It follows from (\ref{gg}) that
\begin{equation}\label{g5}
 |g(k)|\leq M_0\frac{e^{(1+a)|{\rm Im} k|}}{|k|^2}\quad \text{for some }\quad M_0>0.
\end{equation}
Since $q(x)=\tilde{q}(x)$ on $[\frac{a+1}{2},a]$, together with (\ref{nj}), we get
 \begin{equation}\label{eq1}
g(k)\!=\!\int_0^{a}\![\tilde{q}(x)-q(x)]\varphi(x,k)\tilde{\varphi}(x,k)dx=\tilde{\varphi}'(a,k)\varphi(a,k)-\tilde{\varphi}(a,k)\varphi'(a,k).
 \end{equation}
 Note that Eq.(\ref{k0}) with $\eta(1)=1$ and $\eta'(1)=0$ implies that
\begin{equation}\label{bi}
 \varphi(a,k)=y(1,k)\quad \text{and} \quad \varphi'(a,k)=y'(1,k).
\end{equation}
 It yields from (\ref{c9}) that
 \begin{equation*}
   d(k)=\frac{\sin k}{k}\varphi'(a,k)-\varphi(a,k)\cos k=\frac{\sin k}{k}[\varphi'(a,k)-\varphi(a,k)k\cot k],
 \end{equation*}
 which implies
 \begin{equation}\label{sdf}
 \varphi'(a,k)=\frac{k}{\sin k}d(k)+\varphi(a,k)k\cot k.
 \end{equation}
Together with (\ref{sdf}) it follows from (\ref{eq1}) that
 \begin{equation*}\label{eq2}
   \begin{split}
   g(k)
   &=\frac{k}{\sin k}[\varphi(a,k)\tilde{d}(k)-\tilde{\varphi}(a,k){d}(k)]\\
   &=\frac{k E(k)}{\sin k}\gamma\tilde{\gamma}\left[\frac{\varphi(a,k)}{\gamma}-\frac{\tilde{\varphi}(a,k)}{\tilde{\gamma}}\right].
   \end{split}
 \end{equation*}
 Set
 \begin{equation}\label{eq3}
   G(k):=\frac{g(k)}{E(k)}=\frac{k}{\sin k}\gamma\tilde{\gamma}\left[\frac{\varphi(a,k)}{\gamma}-\frac{\tilde{\varphi}(a,k)}{\tilde{\gamma}}\right].
 \end{equation}
 Observing that $d(k)/\gamma=\tilde{d}(k)/\tilde{\gamma}$, one has
 \begin{equation*}
   \frac{1}{\gamma}\left[\frac{\sin k}{k}\varphi'(a,k)-\varphi(a,k)\cos k\right]=\frac{1}{\tilde{\gamma}}\left[\frac{\sin k}{k}\tilde{\varphi}'(a,k)-\tilde{\varphi}(a,k)\cos k\right],
 \end{equation*}
 which implies
 \begin{equation*}
\frac{\varphi(a,n\pi)}{\gamma}-\frac{\tilde{\varphi}(a,n\pi)}{\tilde{\gamma}}=0,\quad n=\pm1,\pm2,\cdot\cdot\cdot,
 \end{equation*}
 and so $G(k)$ is an entire function of $k$ from (\ref{eq3}).

 Due to (\ref{g5}), we know that $G(k)$ satisfies the condition (i) in Lemma \ref{l3}.
From (\ref{k19}) and (\ref{uy}) it follows that
\begin{equation}\label{eqq}
 E(\pm i\tau)=\frac{ce^{(a+1)\tau}}{\tau^{m+3}}[1+o(1)],\quad c\ne0,\quad \tau\to+\infty,
\end{equation}
which implies from (\ref{g5}) and (\ref{eq3}) that
\begin{equation*}
 |G(i\tau)|\le C \tau^{m+1},\quad \tau\to+\infty,
\end{equation*}
where $m\ge0$ appears in Theorem \ref{t1}.
It yields  $\mathop {\varlimsup}\limits_{\tau\to+\infty}{\log|G({\rm i}\tau)|}/{\tau}:=A\le0$. If we can prove $|G(k)|\le C$ for $k\in \mathbb{R}$ (see $(\ast)$ below), then it follows from Lemma \ref{l3} that for all $k\in \overline{\mathbb{C}}_+$
\begin{equation}\label{rq}
 |G(k)|\le C.
\end{equation}
Note that $G(k)$ is even, so Eq.(\ref{rq}) holds on the whole complex plane. This implies that $G(k)$ is a constant from Liouville's theorem. In addition, for the sequence $\{n\pi\}_{n\ge1}$ there holds $G(n\pi)\to0$ as $n\to\infty$ (see $(\ast)$ below). It follows that   $G(k)\equiv0$, which implies $g(k)\equiv0$, and so $q(x)=\tilde{q}(x)$ for $x\in [0,a]$ by Lemma \ref{l0}.

Now, we shall prove $(\ast)$: $G(k)$ is bounded on $\mathbb{R}$ and $G(n\pi)$ tends to zero as $n\to\infty$.
Using (\ref{k8}), (\ref{k9}), (\ref{k11}) and (\ref{uy}), we get
\begin{equation*}
  E(k)=\frac{\sin(k(1-a))}{k\gamma\eta(0)^{1/4}}\left[1+O\left(\frac{1}{k}\right)\right],\quad |k|\to\infty,\quad k\in \mathbb{R},
\end{equation*}
which implies $\gamma\eta(0)^{1/4}$ is uniquely determine by $E(k)$ if $a\ne1$. Substituting (\ref{gg}) into (\ref{eq3}), we have
\begin{equation}\label{yr}
G(k)=\frac{\tilde{\gamma }}{\eta(0)^{1/4}\sin k}\int_0^a\left(K(a,t)-\tilde{K}(a,t)\right)\sin(kt)dt.
\end{equation}
Note that $G(k)$ is an entire function of $k$ from the above argument, thus, zeros of $\sin k$ can not be poles of $G(k)$. Thus, it follows from (\ref{yr}) that
 \begin{equation*}
 \int_0^a\left(K(a,t)-\tilde{K}(a,t)\right)\sin(n\pi t)dt=0,\quad n=0,\pm1,\pm2\cdot\cdot\cdot.
 \end{equation*}
Letting $k\to n\pi$ in (\ref{yr}), we get from the L'Hospital principle that
 \begin{equation}\label{kji}
   G(n\pi)=\frac{\tilde{\gamma }\int_0^a\left(\tilde{K}(a,t)-{K}(a,t)\right)t\cos(n\pi t)dt}{\eta(0)^{1/4}(-1)^n},\quad n=0,\pm1,\pm2\cdot\cdot\cdot.
 \end{equation}
Thus, $G(k)$ is bounded on $\mathbb{R}$ and $G(n\pi)$ tends to zero as $n\to\infty$ from (\ref{kji}). Therefore, we have finished the proof.
\end{proof}
\begin{proof}[Proof of Theorem \ref{t4}]
By a similar argument to the proof of Theorem \ref{t3}, we know that it is enough to show the function $g(k)\equiv0$, where $g(k)$ is defined in (\ref{msz}) with $(a+1)/2$ replacing by $a-b$ (because now $q(x)=\tilde{q}(x)$ on $[a-b,a]$ from (\ref{t40})). From (\ref{eq1})
and (\ref{bi}), together with the boundary condition in (\ref{1}), we get
\begin{equation}\label{gg3}
g(k)=0 \quad\text{for}\quad k\in D\cup\{k_n'\}_{n\ge n_0} .
\end{equation}

Since $|{\rm Im}k|=r|\sin\theta|$, where $k=re^{{\rm i}\theta}$, it follows from (\ref{g5}) with $a+1$ replacing by $2(a-b)$ that
\begin{equation*}
h_g(\theta):=\mathop {\varlimsup }\limits_{r \to \infty }\frac{\log |g(re^{{\rm i}\theta})|}{r}\leq2(a-b)|\sin\theta|,
\end{equation*}
which implies
\begin{equation}\label{z3}
 \frac{1}{2\pi}\int_0^{2\pi}h_g(\theta)d\theta\leq\frac{2(a-b)}{\pi}\int_0^{2\pi}|\sin\theta|d\theta=\frac{4(a-b)}{\pi}.
\end{equation}
On the other hand, from (\ref{gg3}) and (\ref{zx}) we have
\begin{equation*}
  N_g(r)\ge N_D(r) + \frac{2(a-1)r}{\pi}[1+o(1)]=\frac{2(\alpha+a-1) r}{\pi}[1+o(1)],\;r\to\infty.
\end{equation*}
It follows from Lemma \ref{3.1} and (\ref{z3}) that if the entire function $g(k)\not\equiv0$ then
\begin{equation*}
  \frac{2(\alpha+a-1)}{\pi}\le \mathop {\varliminf }\limits_{r \to \infty }  \frac{N_g(r)}{r}\leq\frac{1}{2\pi}\int_0^{2\pi}h_g(\theta)d\theta\le\frac{4(a-b)}{\pi},
\end{equation*}
which yields $\alpha\le a+1-2b$.
However, now $\alpha> a+1-2b$, it yields $g(k)\equiv0$. The proof is complete.
\end{proof}

\noindent\textbf{Appendix}
\vspace*{1.5mm}

Let us give the proof for Proposition \ref{p2.1}. Consider the equation for $\xi$
\begin{equation}\label{ap3}
  \xi-\lambda\left(\frac{\log w}{w}+\frac{\log(1+\xi)}{w}\right)=0.
\end{equation}
where $w$ is fixed with sufficiently large modulus such that $\left|\lambda\frac{\log w}{w}\right|:=\delta<1/4$.
Denote $$f(\xi):=\xi,\quad h(\xi):=-\lambda\left(\frac{\log w}{w}+\frac{\log(1+\xi)}{w}\right).$$
Consider the contour $\Gamma_\delta:=\{\xi\in \mathbb{C}: |\xi|=3\delta\}$ and the disk $D_\delta:=\{\xi\in \mathbb{C}: |\xi|\le 3\delta\}$. Since $\log (1+\xi)$ is bounded for $\xi\in D_\delta$, we can choose $w$ (only need to be sufficiently large) such that
\begin{equation*}
  |h(\xi)|\le |\lambda|\left(\left|\frac{\log w}{w}\right|+\left|\frac{\log(1+\xi)}{w}\right|\right)\le 2\delta,\quad \xi\in D_\delta.
\end{equation*}
It follows that when $\xi\in \Gamma_\delta$, $|f(\xi)|=3\delta>2\delta\ge|h(\xi)|$. Using the Rouch\'{e} theorem, we conclude that $f(\xi)+h(\xi)$ has a unique  (simple) zero inside $\Gamma_\delta$. Therefore, the equation (\ref{ap3}) has a unique solution $\xi=\xi(w)$ for any sufficiently large $|w|$.

For the equation (\ref{k5}), by changing the variable $z=w(1+\xi)$, we can transform it into (\ref{ap3}). Conversely, from (\ref{ap3}), by letting $\xi=z/w-1$, we can get (\ref{k5}). Hence the equation (\ref{k5}) is equivalent to (\ref{ap3}). So Eq.(\ref{k5})  has a unique solution for any sufficiently large $|w|$.

Next, let us prove (\ref{k6}). Using (\ref{ap3}) again, we have
\begin{equation*}
 |\xi|\le C_1\frac{\log |w|}{|w|},\quad \left|\xi-\lambda\frac{\log w}{w}\right|\le C_2  \frac{\log (1+|\xi|) }{|w|}\le  {C_3}\frac{\log |w|}{|w|^2}
\end{equation*}
for sufficiently large $|w|$, where $C_j>0$ ($j=\overline{1,3}$) are  constants. It follows from (\ref{ap3}) and $z=w(1+\xi)$ that
\begin{equation*}
  z=w+\lambda\log w+O\left(\frac{\log |w|}{|w|}\right).
\end{equation*}

\vspace*{2.5mm}

\noindent {\bf Acknowledgments.}
The authors would like to thank the referees for valuable suggestions and comments.
The author Xu was supported in part by the Startup Foundation for Introducing Talent of NUIST. The author Yang was supported in part by the National Natural Science Foundation of
China (11611530682 and 11871031).
The author Buterin was supported in part by by RFBR (Grants 15-01-04864). The authors Buterin and Yurko were supported by the Ministry of Education and
Science of RF (Grant 1.1660.2017/PCh) and by RFBR (16-01-00015 and 17-51-53180).

\end{document}